\documentclass{amsart}
\usepackage{amssymb}
\newtheorem{Theo}{Theorem}
\newtheorem{Lem}[Theo]{Lemma}
\newtheorem{Cor}[Theo]{Corollary}

\newtheorem{Prop}[Theo]{Proposition}
\newtheorem{Prob}{Problem}

\newcommand{\Z}{\mathbb{Z}}
\newcommand{\Aut}{\mathrm{Aut}}
\newcommand{\F}{\mathbb{F}}

\newcommand{\rk}{\mathrm{rk}}
\newcommand{\normal}{\triangleleft}

\begin{document}
\title[Normal and characteristic subgroup growth]{Large normal subgroup growth and large characteristic subgroup growth}
\author[Y. Barnea]{Yiftach Barnea}
\address{
Department of Mathematics, Royal Holloway, University of London, Egham, Surrey, TW20 0EX, United Kingdom}
\email{y.barnea@rhul.ac.uk}
\author[J.-C. Schlage-Puchta]{Jan-Christoph Schlage-Puchta}
\address{
Institut f\"ur Mathematik Ulmenstr. 69 18051 Rostock, Germany}
\email{jan-christoph.schlage-puchta@uni-rostock.de}
\thanks{2010 Mathematics subject classification. 20E07, 20E18, 20F05}
\thanks{Keywords: normal subgroup growth, characteristic subgroup growth, Golod-Shafarevich groups, large groups, profinite groups}
\maketitle
\begin{abstract}
The fastest normal subgroup growth type of a finitely generated group is $n^{\log n}$. Very little is known about groups with this type of growth. In particular, the following is a long standing problem: Let $\Gamma$ be a group and $\Delta$ a subgroup of finite index. Suppose $\Delta$ has normal subgroup growth of type $n^{\log n}$, does $\Gamma$ has normal subgroup growth of type $n^{\log n}$? We give a positive answer in some cases, generalizing a result of M\"uller and the second author and a result of Gerdau. For instance, suppose $G$ is a profinite group and $H$ an open subgroup of $G$. We show that if $H$ is a generalized Golod-Shafarevich group, then $G$ has normal subgroup growth of type of $n^{\log n}$. We also use our methods to show that one can find a group with characteristic subgroup growth of type $n^{\log n}$.
\end{abstract}

\section{Introduction and results}

For a group $\Gamma$ let $s_n^\normal(\Gamma)$ be the number of normal subgroups of $\Gamma$ of index at most $n$. Very little in known about the possible asymptotic behaviour of this sequence, see \cite[Chapter 2 and Section 9.4]{subgroup growth} for background. Lubotzky in \cite{Lub} showed, that for any finitely generated group we have that $s_n^\normal(\Gamma)\ll n^{c\Omega(n)}$ for some constant $c$, where $\Omega(n)$ denotes the number of prime divisors of $n$ counted with multiplicity, and Mann in \cite{Mann} showed that for a non-abelian free group we have that $s_n^\normal(\Gamma)> n^{c\log n}$ for some $c>0$ and infinitely many $n$. Comparing these results we find that the normal subgroup growth of a non-abelian free group is of type $n^{\log n}$. We say that a function $f(n$) is of \textit{type} $g(n)$ if there are constants $c_1, c_2$, such that $\log f(n)\leq c_1\log g(n)$ for all $n$, and there are infinitely many $n$ such that $\log f(n)>c_2\log g(n)$.

Our first theorem, Theorem~\ref{thm:elementary} is quite technical. Therefore, to motivate the reader, we start with Corollary~\ref{Cor:characteristic}  that might have more general interest. It concerns with characteristic subgroup growth. In general the characteristic subgroup growth of a group is even more mysterious than the normal subgroup growth, in particular, determining the characteristic growth of a free non-abelian group appears to be quite difficult. However, the following is a simple consequence of Theorem~\ref{thm:elementary}.
\begin{Cor}
\label{Cor:characteristic}
Let $\Gamma$ be a virtually non-abelian free group with trivial center and suppose $\Gamma$ has a finite outer automorphism group. Then $\Gamma$ has characteristic subgroup growth of type $n^{\log n}$. In particular if $\Gamma=A\ast B$ is the free product of two non-trivial finite groups such that at least one has order greater than 2, then $\Gamma$ has characteristic subgroup growth of type $n^{\log n}$.
\end{Cor}

We would like to thank  H.~Wilton for suggesting the example of free product of two non-trivial finite groups and for L.~Mosher, for pointing us to the literature about its outer automorphisms, see the following question in MathOverflow \cite{MathOverflow big}.

We return now to consider normal subgroup growth. One of the basic problems concerning the normal subgroup growth is the question whether we can compare the normal subgroup growth of a group and a subgroup of finite index. If $\Delta<\Gamma$ is a subgroup of finite index, we can intersect normal subgroups of $\Gamma$ with $\Delta$, that is, if $\Gamma$ has many normal subgroups, so has $\Delta$. More precisely, a slight variation of the proof of \cite[Proposition~1.3.2~(ii)]{subgroup growth} gives us that $s_n^\normal(\Gamma) \leq s_n^\normal(\Delta) n^{(\Gamma:\Delta)}$. Therefore, unless we are aiming at results of high precision, the difficult problem is to decide whether a finite index subgroup of a group $\Gamma$ can have substantially more normal subgroups than $\Gamma$ itself. Lubotzky and Segal in \cite[Problem 4 (a)]{subgroup growth} ask the following.
\begin{Prob}
Let $\Gamma$ be a group and $\Delta$ a subgroup of finite index. Suppose $\Gamma$ has polynomial normal subgroup growth, that is,  type $n$, does $\Delta$ have polynomial normal subgroup growth?
\end{Prob}
The fundamental problem we consider is a slight variation on it.
\begin{Prob}
Let $\Gamma$ be a group and $\Delta$ a subgroup of finite index. Suppose $\Delta$ has normal subgroup growth of type $n^{\log n}$, does $\Gamma$ has normal subgroup growth of type $n^{\log n}$?
\end{Prob}
This basic question seems to be quite hard and unfortunately we are only able to answer it in special cases.
In \cite[Theorem~1]{normal II}, M\"uller and the second author claimed the following theorem. If $\Gamma$ is a finitely generated group and $\Delta$ is a normal subgroup of finite index such that $\Delta$ maps onto a group $G$ such the pro-$p$ completion of $G$ is a non-abelian free pro-$p$ group for some prime $p$ and $p\nmid(\Gamma:\Delta)$, then $\Gamma$ has normal subgroup growth of type $n^{\log n}$. The proof used Mann's construction of large elementary abelian sections and representation theory to study the action of $\Gamma/\Delta$ on these sections. Gerdau in \cite{Victor} claimed that the condition $p\nmid|\Gamma/\Delta|$ is not necessary. In his proof Gerdau replaced ordinary representation theory by modular representation theory. Notice that actually there is no need to require $\Delta$ to be normal. However, a careful examination of both proofs shows that weaker results are proven, namely, if $\Gamma$ is a finitely generated group and $\Delta$ is a normal subgroup of finite index such that the pro-$p$ completion of $\Delta$ is a non-abelian free pro-$p$ group, then $\Gamma$ has normal subgroup growth of type $n^{\log n}$. We do not know whether the original statements are true. The problem in the proofs is that a group that contains a subgroup of finite index which projects onto a free group needs not project onto a virtually free group. An example of such group is the semi-direct product $C_2 \ltimes (F_2\times F_2)$, where $C_2$ acts by interchanging the factors of the direct product. This has a subgroup of order 2, that projects onto $F_2$, but does not project onto any virtually (non-abelian free) group. However, it is still true that all virtually (non-abelian free) groups have normal subgroup growth of type $n^{\log n}$.

In this note we give a couple of far more general results.  Our first result appears rather technical, however, in concrete cases the conditions of the theorem are easy to establish. We write $d(\Gamma)$ for the minimal number of generators of $\Gamma$ and if $\Delta$ is a normal subgroup of $\Gamma$ we write $d_{\Gamma}(\Delta)$ for the minimal number of generators of $\Delta$ as a normal subgroup of $\Gamma$, in the case of topological groups we take topological generators. We call a group a \textit{CMEA group} if it is a central extension of an elementary abelian $p$-group by an elementary abelian $p$-group. We write $\rk_{cm} G$ for the \textit{CMEA rank} of a pro-$p$ group $G$, that is, the logarithm to base $p$ of the maximal order of a CMEA image of $G$.

\begin{Theo}
\label{thm:elementary}
Let $\Gamma$ be a $d$-generated group, $\Delta$ a normal subgroup of finite index in $\Gamma$ and $p$ a prime number.
\begin{enumerate}
\item Let $H$ be the pro-$p$ completion of $\Delta$ and let $\Psi$ be the pre-image of $\Phi(H)$ under the canonical map $\Delta \rightarrow H$. Then $d_{\Gamma}(\Psi)\geq \frac{rk_{cm} H}{(\Gamma:\Delta)}-d$.
\item Let $c>0$ be a real number. Suppose $\Delta_i$ is an infinite sequence of normal subgroup of $\Gamma$ contained in $\Delta$ and of index $p$-power in $\Delta$. If for all $i$ we have that $\rk_{cm} \Delta_i>c(\Gamma:\Delta_i)^2$, then $\Gamma$ has normal subgroup growth of type $n^{\log n}$. In fact, the number of subgroups of $\Delta$, which are normal in $\Gamma$, has growth type $n^{\log n}$.
\end{enumerate}
\end{Theo}

Let us see that Theorem~\ref{thm:elementary} implies Gerdau's result (as proved rather than as claimed). Let $\Gamma$, $\Delta$ and $H$ be as in Gerdau's result, by our assumption $H$ is a free non-abelian pro-$p$ group and let $r \geq 2$ be its number of generators. Write $\widehat{\Gamma}$ and $\widehat{\Delta}$ for the profinite completion of $\Gamma$ and $\Delta$ respectively. Then $\widehat{\Gamma}$ contains $\widehat{\Delta}$ because $\Delta$ is of finite index in $\Gamma$. Write $K$ for the kernel of the map from $\widehat{\Delta}$ to $H$ and trivially, $\widehat{\Delta}/K \cong H$. Clearly, $H$ is the pro-$p$ completion of $\widehat{\Delta}$ so $K$ is a characteristic subgroup of $\widehat{\Delta}$ and therefore normal in $\widehat{\Gamma}$. Then $\Gamma$ maps densely into $G=\widehat{\Gamma}/K$ and $\widehat{\Delta}/K \cong H$ is contained in $G$.

Let $N$ be a normal subgroup of finite index of $G$ contained in $H$ and $\Omega$ its pre-image in $\Delta$. Since $N$ is a free pro-$p$ group its mapped onto the free CMEA-group
\[
\langle x_i, y_i, z_{ik} \mid 1\leq k < i\leq d,  x_i^p=y_i, [x_i,x_k]=z_{ik}, y_i^p=z_i^p=[x_i,y_j]=[x_i, z_{jk}]=1\rangle,
\]
with $d=(H:N)(r-1)+1=(\Delta:\Omega)(r-1)+1$ and hence $\Omega$ is mapped onto the same free CMEA-group.
In particular, we have that $\rk_{cm} \Omega \geq d+\binom{d+1}{2}$. Take $c=\frac{(r-1)^2}{2(\Gamma:\Delta)^2}>0$, which is independent of $N$, then $d = \sqrt{2c}(\Gamma:\Delta)(\Delta:\Omega)+1=\sqrt{2c}(\Gamma:\Omega)+1$. Therefore, $\binom{d+1}{2}>d^2/2 > c(\Gamma:\Omega)^2$. It follows from Theorem~\ref{thm:elementary} part~(2) that $G$ has normal subgroup growth of type $n^{\log n}$.

Next we would like to consider virtual Golod-Shafarevich groups. However, to consider non-normal subgroups we need to be able to change to a normal subgroup of finite index, but subgroups of finite index of Golod-Shafarevich groups are not necessarily Golod-Shafarevich. Thus, we consider the larger class of generalized Golod-Shafarevich groups introduced by Ershov and Jaikin-Zapirain in \cite{GGS}. We postpone the formal definition of a generalized Golod-Shafarevich group to the next section, more detailed treatments can be found in \cite{GGS} and \cite[chapter~5]{survey}.
\begin{Theo}
\label{thm:GS}
Let $G$ be a profinite group, $H$ an open subgroup, which is a generalized Golod-Shafarevich pro-$p$ group. Then $G$ has normal subgroup growth of type $n^{\log n}$.
\end{Theo}

The proof of Theorem~\ref{thm:GS} requires no representation theory, nevertheless, it also implies the result of Gerdau. Indeed, let $\Gamma$ and $\Delta$ be as in Gerdau's result. We take $\widehat{\Gamma}$, $\widehat{\Delta}$, $H$ and $K$ as above. By our assumption $H$ is a non-abelian free pro-$p$ group. As $\Delta$ is normal in $\Gamma$ (we can always assume that by passing to a subgroup of finite index), $\widehat{\Delta}$ is normal in $\widehat{\Gamma}$. Hence, $K$ is normal in $\widehat{\Gamma}$. Let $G=\widehat{\Gamma}/K$ and the result follows since non-abelian free pro-$p$ groups are generalized Golod-Shafarevich groups.

Both Theorems imply that groups of virtual positive $p$-deficiency have normal subgroup growth of type $n^{\log n}$. For the notion of $p$-deficiency we refer the reader to \cite{pdef}. This is easy to see for Theorem~\ref{thm:GS}, since groups of positive $p$-deficiency are virtually Golod-Shafarevich (see \cite[Theorem~5.5]{BT}). For Theorem~\ref{thm:elementary} this requires some computation similar to the case of virtually free groups, which we skip here.

Next we give an example, where Theorem~\ref{thm:elementary} is applicable, but Theorem~\ref{thm:GS} is not.

\begin{Prop}
\label{prop:triabelian}
Let $G$ be the quotient of the free pro-$p$ group $F$ in $d\geq 2$ generators by $\gamma_2(F')$. Then $G$ has normal growth type $n^{\log n}$.
\end{Prop}

This result is sharp in view of a result by Segal who showed in \cite{Dan} that metabelian groups have normal growth of type $n^{(\log n)^{1-\delta}}$ for some $\delta>0$. In other words, abelian by abelian groups cannot have large normal growth, while (nilpotent of class 2) by abelian can.

We do not have an example of a Golod-Shafarevich group which does not satisfy the conditions of Theorem~\ref{thm:elementary}, but we believe that such examples should exist. In particular, if $G$ is a Golod-Shafarevich group of subexponential subgroup growth, then Theorem~\ref{thm:GS} is applicable, but Theorem~\ref{thm:elementary} is not. While it is likely that such groups exist, no examples are known to us.

Finally we remark that both Theorem~\ref{thm:elementary} and \ref{thm:GS} apply to Fuchsian groups of positive hyperbolic volume. Such a group has a normal subgroup of finite index, which is a surface group with at least 4 generators, and both the fact that surface groups map onto large CMEA groups and that the pro-$p$ completion of surface groups is Golod-Shafarevich follow immediately from the definition. This is important, as the normal subgroup growth of Fuchsian groups was used in \cite{quasiplatonic} to count isomorphism types of algebraic curves with many automorphisms, and the corrected version of \cite{normal II} does not suffice for this purpose.

\section{Virtual Golod-Shafarevich groups}
\begin{Lem}
\label{Lem:generators}
A pro-$p$ group $G$ has normal growth of type $n^{\log n}$ if and only if there exists some $c>0$, such that for infinitely many normal subgroups $N$ of finite index we have $d_G(N)>c\log (G:N)$.
\end{Lem}
\begin{proof}
Suppose that $N$ is an open normal subgroup satisfying $d_G(N)>c\log(G:N)$. Then $N/\Phi_G(N)$ is a vector space of dimension $>c\log(G:N)$, thus, this space contains at least $p^{\frac{c^2}{4}\log(G:N)^2}$ subspaces of codimension $d=\left\lfloor \frac{c}{2}\log(G:N) \right\rfloor$. Since every subspace corresponds to a normal subgroup of index $(G:N)p^d$, we conclude that $G$ has normal growth of type $n^{\log n}$.

Suppose on the other hand that there exists a function $f(n)$, $f(n)=o(n)$, such that $d_G(N)\leq f(n)$ for all $N\normal G$ with $(G:N)=p^n$. Then for each normal subgroup $N$ with $(G:N)=p^n$ there exists a sequence $G=N_0>N_1>N_2>\dots>N_n=N$ with $N_i\normal G$ and $(N_i:N_{i+1})=p$. If $N_i$ is fixed, $N_{i+1}$ can be chosen in $\frac{p^{d_G(N_i)-1}}{p-1}$ ways, hence the number of normal subgroups of index $p^n$ is at most $p^{\sum_{\nu\leq n} f(\nu)} = p^{o(n^2)}$, and we conclude that the normal growth of $G$ is not of type $n^{\log n}$.
\end{proof}

\begin{Lem}
\label{Lem:virtual}
Let $G$ be a profinite group, which contains an open normal subgroup $H$ that is a pro-$p$ group. If there exists some $c>0$ such that $H$ contains infinitely many characteristic open subgroups $N$ satisfying $d_H(N)>c\log(H:N)$, then $G$ has normal growth $n^{\log n}$
\end{Lem}
\begin{proof}
Suppose that $N$ is a characteristic subgroup of $H$ satisfying $d_H(N)>c\log(H:N)$. Then $N$ is normal in $G$, let $x_1, \ldots, x_d$ be elements generating $N$ as a normal subgroup of $G$. Let $g_1, \ldots, g_k$ be representatives of the cosets $G/H$. Then $x_i^{g_j}$, $1\leq i\leq d$, $1\leq j\leq k$ generates $N$ as a normal subgroup of $H$, hence $(G:H)d\geq d_H(N)$. We conclude that
\begin{multline*}
d_G(N)\geq \frac{d_H(N)}{(G:H)}\geq \frac{c}{(G:H)}\log(H:N)\\
 = \frac{c}{(G:H)}\left(\log(G:N)-\log(G:H)\right)\geq \frac{c}{2(G:H)}\log(G:N),
\end{multline*}
provided that $(G:H)(H:N)=(G:N) > (G:H)^2$, that is, $(H:N)>(G:H)$. The latter condition excludes only finitely many $N$, and our claim follows.
\end{proof}

Let $F(X)$ be the free pro-$p$ group over the set $X$. Let $\F_p\langle\langle X\rangle\rangle$ be the ring of power series in non-commuting variables over $X$. The map $x \mapsto 1+x$ for all $x \in X$ extends to the Magnus map $\mu:F(X)\rightarrow\F_p\langle\langle X\rangle\rangle$. Magnus in \cite{Magnus} proved that it is injective. To each element $x\in X$ we associate a positive integer  $d_x$, which we call the degree of $x$. We can extend this degree to a function $D:\F_p\langle\langle X\rangle\rangle\rightarrow\mathbb{N}$ by defining the degree of a monomial to be the sum of the degrees of its factors, and the degree of a linear combination of monomials as the maximal degree of one of these monomials. Now define a degree function $d:F(X)\rightarrow\mathbb{N}$ by putting $d(w)=D(\mu(w)-1)$. It is not hard to see the an element in the $n$-dimension subgroup has degree at least $n$. For a set $A \subset F(X)$ define the Hilbert series $H_{d, A}(t)=\sum_{a\in A} t^{d(a)}$.

Let $\langle X | R\rangle$ be a presentation with $X$ finite. We say that this presentation is a generalized Golod-Shafarevich presentation, if there exists a degree function $d$ and a real number $t_0\in (0,1)$, such that $1-H_{d, X}(t_0)+H_{d,R}(t_0)<0$. A pro-$p$ group is called a generalized Golod-Shafarevich group, if it has a generalized Golod-Shafarevich presentation. Note that the usual notion of a Golod-Shafarevich group corresponds to the degree function which assigns to all $x\in X$ the degree 1. We use the same terminology as in \cite{survey}. The following is contained in \cite[Theorem~4.3 and Theorem~4.4]{survey}.
\begin{Prop}
\label{Prop:Golod-Shafarevich properties}
\begin{enumerate}
\item A generalized Golod-Shafarevich group is infinite.
\item An open subgroup of a generalized Golod-Shafarevich group is again generalized Golod-Shafarevich.
\end{enumerate}
\end{Prop}

In \cite[appendix]{Andrej} Jaikin-Zapirain showed that a finitely generated generalized Golod-Shafarevich group has subgroup growth of type at least $p^{n^\beta}$ for some $\beta>0$. Using the basic idea from his proof we show the following lemma. For that we need to recall that given a pro-$p$ group $G$ the $n$-dimension subgroups of $G$ is $$D_n(G)=\prod_{ip^j \geq n} \gamma_i(G)^{p^j}.$$
\begin{Lem}
\label{Lem:Golod-Shafarevich}
Let $H$ be a finitely generated generalized Golod-Shafarevich pro-$p$ group. Then there exists $c>0$, such that $H$ contains infinitely many open characteristic subgroups $N$ satisfying $d_H(N)>c\log(H:N)$.
\end{Lem}
\begin{proof}
Pick a presentation $\langle X|R\rangle$ of $H$, a degree function $d$, and a real number $t_0\in(0,1)$ such that $1-H_{d, X}(t_0)+H_{d, R}(t_0)=-\delta<0$. 
Our aim is to show that $D_n=D_n(H)$ needs many generators as normal subgroups of $H$. Since $\Phi_H(D_n)\leq D_{n+1}$ we have that $d_H(D_n)\geq\log(D_n:D_{n+1})$.

Let $g_1, \dots, g_m$ be elements of $H$, which generate $D_n$ as a normal subgroup. Then there exist elements $r_1, \ldots, r_m\in F(X)$, where $F(X)$ is the free pro-$p$ group on the set $X$, such that $r_i$ maps to $g_i$ under the map $F(X)\rightarrow H$ induced by the presentation. The degree of $r_i$ is at least $n$, since $r_i\in D_n$. Put $\widetilde{R}=R\cup\{r_1, \ldots, r_m\}$. Then $\langle X|\widetilde{R}\rangle = H/D_n$ is a finite group, in particular, this group does not satisfy the Golod-Shafarevich inequality. The Hilbert series of this finite group is $H_R(t)+\sum_{i=1}^m t^{d(r_i)}$, hence
\[
0\leq 1-H_{d,x}(t_0)+H_{\widetilde{R}}(t_0) \leq 1-H_{d,X}(t_0)+H_R(t_0) + mt_0^n = mt_0^n-\delta.
\]
We conclude that $m>\delta t_0^{-n}$, and therefore $m>a^n$ for some $a>1$, provided that $n$ is sufficiently large.

Combining our estimates we obtain
\[
d_H(D_n)\geq\max\big(a^n, \log(D_n:D_{n+1})\big).
\]
Suppose first that $\alpha=\limsup\frac{\log(D_n:D_{n+1})}{a^n}$ is finite. Then for all $n$ there exists a constant $A$ such that
\[
\log(H:D_n) = \sum_{\nu=1}^{n-1}\log(D_\nu:D_{\nu+1})\leq A+(\alpha+1)\sum_{\nu=1}^{n-1} a^\nu \leq A+\frac{\alpha+1}{a-1}a^n.
\]
Since $D_n$ is a characteristic subgroup of $H$ and $d_H(D_n)\geq a^n$ we deduce our claim.

On the other hand, if $\limsup\frac{\log(D_n:D_{n+1})}{a^n}$ is infinite, then we can pick a subsequence $(n_i)$, such that for all $i$ and  for all $m<n_i$ we have that $$\frac{\log(D_{n_i}:D_{n_i+1})}{a^{n_i}}>\frac{\log(D_m:D_{m+1})}{a^m}.$$ For such $n_i$ we have that
\begin{multline*}
\log(H:D_{n_i}) = \sum_{\nu=1}^{n_i-1} \log(D_\nu:D_{\nu+1}) \leq \log(D_{n_i}:D_{n_i+1})\sum_{\nu=1}^{n_i-1} a^{\nu-n_i}\\
 \leq \frac{1}{a-1}\log(D_{n_i}:D_{n_i+1}) \leq \frac{1}{a-1} d_H(D_{n_i}).
\end{multline*}
Again, our claim follows.
\end{proof}

We can now prove Theorem~\ref{thm:GS}.
\begin{proof}[Proof of Theorem~\ref{thm:GS}] Let $G, H$ be as in the theorem. By Proposition~\ref{Prop:Golod-Shafarevich properties} (2) we may replace $H$ by an open normal subgroup, which is still generalized Golod-Shafarevich. Then Lemma~\ref{Lem:Golod-Shafarevich} implies that $H$ satisfies the assumptions of Lemma~\ref{Lem:virtual}, and we conclude that $G$ has normal subgroup growth of type $n^{\log n}$.
\end{proof}

\section{Groups with large meta-(elementary abelian) sections}

Let us prove part~1 of Theorem~\ref{thm:elementary}.
\begin{proof}[Proof of Theorem~\ref{thm:elementary}, part~1] Recall that $\Gamma$ is a $d$-generated group, $\Delta$ a normal subgroup of $\Gamma$ of finite index, $p$ a prime number, $H$ the pro-$p$ completion of $\Delta$, and $\Psi$ is the pre-image of $\Phi(H)=H^p[H,H]$, the Frattini subgroup of $H$, under the canonical map $\Delta \rightarrow H$. We write $L=\Phi_{H}(\Phi(H))=\Phi(H)^p[\Phi(H),H]$ for the normal Frattini subgroup of $\Phi(H)$ in $H$ and $\Lambda$ for its preimage in $\Delta$.

Clearly $\Psi$ and $\Lambda$ are characteristic subgroups in $\Delta$ and thus normal in $\Gamma$.  Suppose that $x_1, \ldots, x_d$ are generators of $\Psi$ as a normal subgroup of $\Gamma$. Then $x_1\Lambda, \ldots, x_d\Lambda$ are generators of $\Psi/\Lambda$ as a normal subgroup of $\Gamma/\Lambda$. Hence, $X = \bigcup \left(x_i\Lambda \right)^{\Gamma/\Lambda}$ generates $\Psi/\Lambda$ as a subgroup. Since $\Psi/\Lambda$ is central in $\Delta/\Lambda$, we have $\Delta/\Lambda \leq C_{\Gamma/\Lambda}(x_i\Lambda)$, which implies
\[
\left| \left(x_i\Lambda \right)^{\Gamma/\Lambda} \right| = \left(\Gamma/\Lambda:C_{\Gamma/\Lambda}(x_i\Lambda) \right)\leq (\Gamma/\Lambda:\Delta/\Lambda)=(\Gamma:\Delta).
\]
We deduce that
\[
d_{\Gamma}(\Psi) \geq d_{\Gamma/\Lambda}(\Psi/\Lambda) \geq \frac{d(\Psi/\Lambda)}{(\Gamma:\Delta)}.
\]

Notice that $H/L$ is the maximal CMEA-quotient of $H$. Because $\Psi/\Lambda$ is an elementary abelian $p$-group we have that
\begin{multline*}
d_{\Gamma}(\Psi) \geq \frac{d(\Psi/\Lambda)}{(\Gamma:\Delta)}=\frac{\log(\Psi:\Lambda)}{(\Gamma:\Delta)}=
\frac{\log(\Phi(H):L)}{(\Gamma:\Delta)}= \\ \frac{\log(H:L)-\log(H:\Phi(H))}{(\Gamma:\Delta)}=
\frac{\rk_{cm}H - d(H)}{(\Gamma:\Delta)} \geq \frac{\rk_{cm}H - d(\Delta)}{(\Gamma:\Delta)} \geq
\\ \frac{\rk_{cm}H - d(\Gamma)(\Gamma:\Delta))}{(\Gamma:\Delta)}
= \frac{\rk_{cm}H}{(\Gamma:\Delta)} - d,
\end{multline*}
and the first part of Theorem~\ref{thm:elementary} follows.
\end{proof}

To prove part~2 of the theorem we need the following lemma from Gerdau \cite{Victor}. Since it has not been published we include the proof here. Note that the statement is quite easy, if $p\nmid |G|$, for in that case every module decomposes into a sum of simple modules and there are only finitely  many isomorphism types of simple modules of $\mathbb{F}_pG$.
\begin{Lem}
\label{Lem:module}
Let $G$ be a finite group and $p$ a prime number. Then there exists a constant $c>0$ depending only on $G$ and $p$, such that for all $\mathbb{F}_pG$ modules $M$ of $\mathbb{F}_p$-dimension $d$ there exist submodules $M_1<M_2<M$, such that $M_2/M_1$ is the direct sum of at least $cd$ isomorphic simple modules.
\end{Lem}
\begin{proof}
For a module $M$ define $\ell(M)$ as the maximal length of an ascending series of submodules in a cyclic submodule of $M$.
Since a cyclic $\F_pG$-module has dimension at most $|G|$, we have that $\ell(M) \leq |G|$. Let $c_1=1/\left(\sum \dim S\right)$, where $S$ runs over all isomorphism classes of simple modules. We prove by induction on $\ell=\ell(M)$ that there exists such a constant $c_\ell>0$ and we take $c$ to be the minimal amongst all the $c_{\ell}$.

Define $N$ to be the submodule of $M$ generated by all simple submodules of $M$. Recall that $N$ is a direct sum of simple modules. Note that any simple module is cyclic and therefore, there are only finitely many isomorphism classes of simple modules. Hence, we have that $N$ is the direct sum of at least $c_1\dim_{\F_p} N$ simple modules of the same isomorphism class. Because a simple module is cyclic we have that $N$ contains a direct sum of at least $c_1\dim_{\F_p} N$ cyclic modules of the same isomorphism class.

For $\ell=1$ we have that all cyclic modules are simple, thus, $M=N$ and the statement follows from the above paragraph. We continue by induction on $\ell$. If $\dim_{\F_p} N>\frac{1}{2}\dim_{\F_p} M$, then $N$ contains a direct sum of $c_1\dim_{\F_p} N\geq \frac{c_1}{2}\dim_{\F_p} M$ isomorphic simple modules. If $\dim_{\F_p} N \leq \frac{1}{2}\dim_{\F_p} M$, then $\ell(M/N)\leq\ell(M)-1$, as the first module in an ascending series of maximal length must be a simple module. By the induction hypothesis $M/N$ contains submodules $M_1N<M_2N$, such that $M_2N/M_1N$ is the direct sum of at least $c_{\ell-1}\dim_{\F_p} M/N$ isomorphic simple modules. We conclude that our claim holds with $c_{\ell}=\frac{1}{2}\min(c_1, c_{\ell-1})$.
\end{proof}

Just as in the case of vector spaces, direct sums of isomorphic modules have many submodules. The following is \cite[Lemma~1]{normal II}.

\begin{Lem}
\label{Lem:counting submodules}
Let $G$ be a finite group, $M$ an $\F_pG$-module. Then there is some $c>0$, such that $M^n$ contains at least $e^{c n^2}$ submodules.
\end{Lem}

We are now ready to prove part~2 of Theorem~\ref{thm:elementary}.
\begin{proof}[Proof of Theorem~\ref{thm:elementary}, part~2] Let $\Delta$, $H$ and $\Delta_i$ be as in the theorem. Let $H_i$ be the closure of the image of $\Delta_i$ in $H$, $\Psi_i$ be the pre-image of $\Phi(H_i)$ in $\Delta$, and $\Lambda_i$ be the pre-image of $\Phi_H(\Phi(H))$.

Let $\Omega$ be the kernel of the map $\Delta\rightarrow H$. Then $\Psi_i$ is generated by $\Omega$, $[\Delta, \Delta_i]$ and $\Delta_i^p$. Since all three groups are normal in $\Gamma$, $\Lambda_i$ is also normal in $\Gamma$ and similarly $\Lambda_i$ is normal in $\Gamma$. Hence, $\Gamma$ acts on $\Psi_i/\Lambda_i$ by conjugation. Because $\Delta$ acts trivial on $\Psi_i/\Lambda_i$, we have that $\Psi_i/\Lambda_i$ is an $\F_p(\Gamma/\Delta)$-module. From Lemma~\ref{Lem:module} and Lemma~\ref{Lem:counting submodules} it follows that there is some $c'>0$, depending only on $\Gamma/\Delta$ and $p$, such that $\Psi_i/\Lambda_i$ has at least $e^{c'(\dim_{\F_p}\Psi_i/\Lambda_i)^2}$ submodules. Since there is a bijection between submodules and normal subgroups of $\Gamma$, which are contained in the interval $(\Psi_i, \Lambda_i)$, we conclude that $\Gamma$ has at least $e^{c'(\dim_{\F_p}\Psi_i/\Lambda_i)^2}$ normal subgroups of index at most $(\Gamma:\Lambda_i)$. Therefore, to prove the theorem it suffices to show that there exists some $\delta>0$, independent of $i$, such that $\dim_{\F_p}\Psi_i/\Lambda_i \geq \delta \log (\Gamma:\Lambda_i)$.

As in the proof of part~1 we have
\[
\log(\Phi(H_i):\Phi_{H_i}(\Phi(H_i))) = \rk_{cm}H_i - d(H_i) \geq \rk_{cm}H_i - d(\Delta_i) \geq\rk_{cm}H_i - d(\Gamma)(\Gamma:\Delta_i).
\]
It follows from the assumption in the theorem that
$$
\log(\Psi_i:\Lambda_i) = \log(\Phi(H_i):\Phi_H(\Phi(H_i))) \geq c(\Gamma:\Delta_i)^2- d(\Gamma)(\Gamma:\Delta_i)\geq c(\Gamma:\Delta_i)- d(\Gamma),
$$
so
$$(\Gamma:\Delta_i) \leq \frac{\log(\Psi_i:\Lambda_i) + d(\Gamma)}{c}.$$
We deduce that
\begin{eqnarray*}
\log (\Gamma:\Lambda_i) & = & \log(\Gamma:\Delta_i)+\log(\Delta_i:\Psi_i)+\log(\Psi_i:\Lambda_i)\\
 & \leq & \log\left(\frac{\log(\Psi_i:\Lambda_i) + d(\Gamma)}{c}\right) + d(H_i) + \log(\Psi_i:\Lambda_i)\\
 & \leq & \log(\Psi_i:\Lambda_i) + d(H)(H:H_i)  + \log(\Psi_i:\Lambda_i)\\
 & \leq & d(H)(\Gamma:\Delta_i)  + 2\log(\Psi_i:\Lambda_i)\\
 & \leq & \frac{d(H)}{c}\log(\Psi_i:\Lambda_i) + \frac{d(H)d(\Gamma)}{c} + 2\log(\Psi_i:\Lambda_i)\\
 & \leq & \left(\frac{d(H)}{c}+3\right)\log(\Psi_i:\Lambda_i),
\end{eqnarray*}
provided that $(\Psi_i:\Lambda_i)$ is sufficiently large. Putting $\delta=\left(\frac{d(H)}{c}+3\right)^{-1}$,
we have that $\dim_{\F_p}\Psi_i/\Lambda_i = \log(\Psi_i:\Lambda_i) \geq \delta \log (\Gamma:\Lambda_i)$ as required, and part~1 of Theorem~\ref{thm:elementary} is proven.
\end{proof}

\begin{proof}[Proof of Proposition~\ref{prop:triabelian}] Let $\widehat{F}_d$ be the free pro-$p$ group of rank $d$. Clearly the normal growth of $\widehat{F}/\gamma_2(\widehat{F}_d')$ is a lower bound for the normal growth of $F_d/\gamma_2(F_d')$. Let $N$ be the normal subgroup of $\widehat{F}_d$ such that $\widehat{F}_d/N \cong (\Z/p^k\Z)^d$. Then $N  \geq \widehat{F}_d'$ and therefore, $\Phi(N) \geq \widehat{F}_d''=\gamma_1(\widehat{F}_d')$ and also $\Phi_N(\Phi(N))\geq \gamma_2(\widehat{F}_d')$. Thus, $(\widehat{F}_d:N)$, $(N:\Phi(N))$, and $\Phi(N)/\Phi_N(\Phi(N))$ do not change when we factor by $\gamma_2(\widehat{F}_d')$. Let $d_1=p^k(d-1)+1$, then $N\cong \widehat{F}_{d_1}$. Since $N$ is a free pro-$p$ group it surjects onto the free CMEA group with $d_1$ generators. The latter has order $p^{d_1+\binom{d_1+1}{2}}>p^{d_1^2/2}$, so $rk_{cm} N> \frac{d_1^2}{2}$. Hence,
\[
rk_{cm} N >\frac{d_1^2}{2}= \frac{1}{2}\left((d-1)(\widehat{F}_d:N)+1\right)^2 > \frac{(d-1)^2}{2}(\widehat{F}_d:N)^2,
\]
and we can apply Theorem~\ref{thm:elementary}. On the other hand $F/\gamma_2(F_d')$ is not Golod-Shafarevich, since, as Zelmanov showed in \cite{Ze}, Golod-Shafarevich groups always contain non-abelian free pro-$p$ subgroups, which $F/\gamma_2(F_d')$ clearly does not.
\end{proof}

\begin{proof}[Proof of Corollary~\ref{Cor:characteristic}]  Let $\Gamma$ be as in the corollary. Since the center of $\Gamma$ is trivial, we can view $\Gamma$ as a subgroup of $\Aut(\Gamma)$. A subgroup of $\Gamma$ is characteristic in $\Gamma$ if and only if it is normal in $\Aut(\Gamma)$. Since $(\Aut(\Gamma):\Gamma)$ is finite, we can apply Theorem~\ref{thm:elementary} to the pair $\Aut(\Gamma)>\Gamma$ in place of $\Gamma>\Delta$, and find that $\Gamma$ has characteristic subgroup growth of type $n^{\log n}$.

It remains to show that free products of finite groups have trivial center and finite outer automorphism group. The statement about the center is easy. In fact, as non-abelian free groups have trivial center, every element of the center of $A \ast B$ is contained in a conjugate of $A$ or $B$. But elements in conjugates of $A$ do not commute with elements in conjugates of $B$, hence the center is trivial. The finiteness of the outer automorphism group was shown by Pettet in \cite[Proposition~2.5]{Pettet}.
\end{proof}

\section{Problems and remarks}

Lemma~\ref{Lem:virtual} falls short of our expectations in two aspects. First, it is unfortunate that $N$ are required to be characteristic subgroups of $H$ rather than normal. We therefore ask the following.

\begin{Prob}
Does there exist a pro-$p$ group $G$, such that for some $c>0$ there exist infinitely many open normal subgroups $N$ with $d_G(N)\geq c\log (G:N)$, but $\lim\frac{d_G(C)}{\log (G:C)}=0$, as $C$ ranges over characteristic open subgroups?
\end{Prob}

The second problem is that Lemma~\ref{Lem:virtual} only gives information on normal growth of type $n^{\log n}$ because our counting methods are very crude. To give an upper bound for the number of normal subgroups one considers all chains of the form $G>N_1>\dots>N_n$, disregarding the facts that $N_n$ might be contained in many different chains, and that not all normal subgroups need the maximal number of generators. For the lower bound one considers one normal subgroup $N$ which needs many generators and counts only normal subgroups between $N$ and $\Phi_G(N)$. Both estimates appear heavily wasteful, but for growth types $n^{\log n}$ and larger the difference is actually quite small. However, if $$\displaystyle{\max_{(G:N)=n}d_G(N)\approx f(\log n)}$$ with $f(k)=o(k)$, one only gets a lower bound of type $e^{f(\log n)^2}$ and an upper bound $e^{f(\log n)\log n}$. In particular, we do not know whether an analogue of Lemma~\ref{Lem:generators} exist. We therefore ask the following.

\begin{Prob}
Do there exist pro-$p$ groups $G, H$, such that for all sufficiently large $n$ we have $$\displaystyle{\max_{(G:N)\leq p^k}d_G(N) > \max_{(H:K)\leq p^k}d_H(K)},$$ but $$\frac{\log s_{p^k}^{\normal}(G)}{\log s_{p^k}^{\normal}(H)} \underset{k \to \infty}{\longrightarrow} 0?$$
\end{Prob}

 In addition, we believe that some generalization of Theorem~\ref{thm:elementary} should be true, although we are not sure what form it should take. Informally, Theorem~\ref{thm:elementary} assumes that $G$ contains infinitely many normal subgroups, which map onto CMEA groups of "size comparable to the corresponding quantity in a free group". If we replace "CMEA" by "elementary abelian" in this statement, we obtain positive upper rank gradient which is equivalent to exponential subgroup growrh. Here we define $\mathrm{rg}^+$ and $\mathrm{rg}^-$ the \textit {upper and the lower rank gradient} of a pro-$p$ group $G$ respectively as
 \[
 \mathrm{rg}^+ = \sup_{G>U_1>U_2>\dots} \lim_{n\rightarrow\infty} \frac{d(U_i)}{(G:U_i)},\qquad\mathrm{rg}^- = \inf_{G>U_1>U_2>\dots} \lim_{n\rightarrow\infty} \frac{d(U_i)}{(G:U_i)},
 \]
 where infimum and supremum are taken over all descending chains of finite index subgroups.

 However, positive upper rank gradient does not imply large normal growth, for example, $G$, the pro-$p$ completion of the restricted wreath product $\Gamma=\Z\wr\F_p$, has polynomial normal growth as shown in the example after Theorem~9.2 in \cite{subgroup growth}, however, let us see that it has positive upper rank gradient.

 Let $\Sigma$ be the base group of $\Gamma$. We have a map $\psi_k:\Gamma \rightarrow P_k=C_{p^k}\wr\F_p$. Let $B$ be the base group of $P_k$ and let $\Delta$ be the pre-image of $B$ and $N$ its closure in $G$. Then $\Delta$ is a normal subgroup of $\Gamma$ of index $p^k$, $\Sigma \leq \Delta$, and $[\Delta,\Delta]\Delta^p$ is in the kernel of $\psi_k$ because $B$ is an elementary abelian $p$-group. Hence,
 $$(N:\Phi(N)) \geq (\Delta:[\Delta,\Delta]\Delta^p) \geq (\Delta \cap \Sigma:[\Delta,\Delta]\Delta^p \cap \Sigma) \geq (\Sigma:\ker \psi_k  \cap \Sigma) = |B|=p^{p^k},$$
 and we obtain that $d(N) \geq p^k=[\Gamma:\Delta]=[G:N]$. In particular, the upper rank gradient of $G$ is as large as the rank gradient of a free group with 2 generators, and the subgroup growth of $G$ is exponential.

One reason that explains the discrepancy between the large subgroup growth and the moderate normal growth is the fact that if a subgroup  $U$ requires many generators, then it contains essentially the whole base group, while $\Phi(U)$ cuts deeply into the base group and needs only few generators. In particular, there are no large CMEA sections in $G$, because most subgroups need only few generators. It would be interesting to know what happens if we circumvent this obstacle by assuming that all subgroups need many generators. We therefore ask the following problem.

\begin{Prob}
If $G$ has positive lower rank gradient, does $G$ necessarily have large normal growth?
\end{Prob}

It is not hard to see that the conditions in Theorem~\ref{thm:elementary} imply exponential subgroup growth. Thus, comparing it to Theorem~\ref{thm:GS} naturally leads to the following problem.

\begin{Prob}
Can you give an example of a generalized Golod-Shafarevich group with subexponential subgroup growth? If not, can you give some other example of a generalized Golod-Shafarevich group to which Theorem~\ref{thm:elementary} cannot be applied? More generally, does the assumption that $G$ has normal subgroup growth of type $n^{\log n}$ imply any other largeness properties?
\end{Prob}

The last question is probably very difficult, as for example the pro-2 completion of Grigorchuk's group or the Nottingham group contain non-abelian free pro-$p$ subgroups, but have only a bounded number of normal subgroups of any given index. Finally we would like to know the answer to the original claim from \cite{normal II}.

\begin{Prob}
Let $\Gamma$ be a large group, that is, it contains a finite index subgroup $\Delta$, which projects onto a non-abelian free group $F$. Does $\Gamma$ have normal subgroup growth of type $n^{\log n}$?
\end{Prob}

Note that a negative answer to Problem~2 would immediately resolve this question in the positive.

For characteristic subgroup growth the first question one might consider is the following.
\begin{Prob}
What is the characteristic subgroup growth of a non-abelian free group?
\end{Prob}
This question was posed by I.~Rivin on MathOverflow, see \cite{MathOverflow characteristic}, and seems intriguingly difficult. W.~Thurston gave an argument that the growth should be similar to the characteristic subgroup growth of $\Z^n$. However, the direct approach leads to well-known open problems on the product replacement algorithm. A first step in this direction could be the following probelm, which is due to Lubotzky.
\begin{Prob}
Is there a characteristic subgroup $C$ in $F_3$, such that $F_3/C$ is a finite simple group?
\end{Prob}

\end{document}